\documentclass[fontsize=11pt,english,headinclude,bibliography=totoc]{amsart} 
\usepackage{ucs}
\usepackage{ae}
\usepackage{microtype}
\usepackage[utf8]{inputenc}
\usepackage[T1]{fontenc}
\usepackage[english]{babel}

\title[The Abel--Jacobi map over the twistor-$\P^1$]{The Abel--Jacobi map over the twistor-$\P^1$ \\ and real local class field theory}
\author{Saverio Caleca and Maximilian Hauck}
\date{}

\usepackage{amsmath}
\usepackage{amssymb}
\usepackage{amstext}
\usepackage{amsthm}

\usepackage{enumitem}

\usepackage{scrlayer-scrpage}
\usepackage{geometry}
\usepackage{xcolor}

\usepackage{mathtools}

\usepackage{mathscinet}
\usepackage[numbers]{natbib}

\usepackage{graphicx}
\usepackage{tikz-cd}
\usepackage{caption}

\usepackage{tabularx}

\usepackage{hyperref}
\usepackage[capitalize]{cleveref}

\newif\ifcomment

\newlength{\currentparindent}
\makeatletter
\newcommand{\@minipagerestore}{\setlength{\parindent}{\currentparindent}}
\makeatother

\newcommand{\nospacepunct}[1]{\makebox[0pt][l]{#1}}

\usepackage{relsize}
\usepackage[bbgreekl]{mathbbol}
\usepackage{amsfonts}

\usepackage{accents}

\DeclareSymbolFontAlphabet{\mathbb}{AMSb}
\DeclareSymbolFontAlphabet{\mathbbl}{bbold}

\DeclareMathOperator{\id}{id}

\DeclareMathOperator{\Hom}{Hom}

\DeclareMathOperator{\Spec}{Spec}

\DeclareMathOperator{\Pic}{Pic}

\DeclareMathOperator{\Perf}{Perf}

\DeclareMathOperator{\Map}{Map}

\DeclareMathOperator{\Sym}{Sym}

\DeclareMathOperator{\Nil}{Nil}
\DeclareMathOperator{\AnSpec}{AnSpec}

\DeclareMathOperator{\Cont}{Cont}

\DeclareMathOperator{\AJ}{AJ}
\DeclareMathOperator{\Proj}{Proj}

\let\Map\map

\newcommand{\colim}{\operatornamewithlimits{colim}}
\newcommand{\Lim}{\operatornamewithlimits{lim}}
\let\lim\Lim

\newcommand\A{\mathbb{A}}

\newcommand\G{\mathbb{G}}

\newcommand\Z{\mathbb{Z}}

\newcommand\R{\mathbb{R}}
\newcommand\C{\mathbb{C}}

\let\P\PP

\newcommand\D{\mathcal{D}}

\let\O\cO

\let\L\cL

\newcommand\BC{\mathcal{BC}}

\newcommand\dR{\mathrm{dR}}

\newcommand\HT{\mathrm{HT}}

\newcommand\can{\mathrm{can}}

\newcommand\an{\mathrm{an}}
\newcommand\Div{\mathrm{Div}}

\newcommand\FF{\mathrm{FF}}

\newcommand\la{\mathrm{la}}

\newcommand\Betti{\mathrm{Betti}}

\newcommand\ab{\mathrm{ab}}
\newcommand\gas{\mathrm{gas}}
\newcommand\TotDisc{\mathrm{TotDisc}}
\newcommand\Ani{\mathrm{Ani}}
\newcommand\Set{\mathrm{Set}}
\newcommand\TDStack{\mathrm{TDStack}}

\let\epsilon\varepsilon
\let\phi\varphi
\let\ol\overline
\let\ul\underline
\let\tensor\otimes

\let\cal\mathcal

\newtheorem{thm}{Theorem}[section]
\newtheorem{prop}[thm]{Proposition}
\newtheorem{lem}[thm]{Lemma}
\newtheorem{cor}[thm]{Corollary}

\theoremstyle{definition}
\newtheorem{defi}[thm]{Definition}

\newtheorem{rem}[thm]{Remark}

\newtheorem{exx}[thm]{Example}

\AddToHook{env/prop/begin}{\crefalias{thm}{prop}}
\AddToHook{env/lem/begin}{\crefalias{thm}{lem}}
\AddToHook{env/cor/begin}{\crefalias{thm}{cor}}
\AddToHook{env/defi/begin}{\crefalias{thm}{defi}}
\AddToHook{env/thmdefi/begin}{\crefalias{thm}{thmdefi}}
\AddToHook{env/lemdefi/begin}{\crefalias{thm}{lemdefi}}
\AddToHook{env/conv/begin}{\crefalias{thm}{conv}}
\AddToHook{env/rem/begin}{\crefalias{thm}{rem}}
\AddToHook{env/warn/begin}{\crefalias{thm}{warn}}
\AddToHook{env/exx/begin}{\crefalias{thm}{exx}}

\numberwithin{equation}{section}

\begin{document}

\begin{abstract}
We study the Abel--Jacobi map over the twistor-$\P^1$ in the context of Scholze's geometrisation of the real local Langlands correspondence. In a similar spirit to a result of Fargues over the Fargues--Fontaine curve, we prove that pullback along the Abel--Jacobi map induces an equivalence on Picard groupoids and use this to recover local class field theory for archimedean local fields.
\end{abstract}

\maketitle

\tableofcontents

\section{Introduction}
 
An important step towards the categorical local Langlands conjecture as formulated by Fargues--Scholze in \cite{FarguesScholze} was Fargues' observation that Deligne's  geometric proof of unramified global class field theory, when carried over to the Fargues--Fontaine curve, yields local class field theory. More recently, by adapting techniques introduced by Camargo in \cite{camargo2024analyticrhamstackrigid}, Scholze made significant progress towards a geometrisation of the local Langlands correspondence over archimedean local fields. In particular, in \cite{RealLLC}, he introduced the key definition of families of twistor-$\P^1$s, which are the archimedean analogue of relative Fargues--Fontaine curves.

The purpose of this article is to study the Abel--Jacobi map over Scholze's twistor.
Our main result is the following:

\begin{thm}[\cref{thm:main}]
\label{thm:intro-main}
Let $E$ be an archimedean local field. Pullback along the map $\AJ$:
\begin{equation}
    \AJ :\Div^1_{E}\to \Pic^1_{E}
\end{equation}
induces an equivalence
\begin{equation*}
    \Map(\Pic^1_E, */\G_m^\an)\xrightarrow{\cong} \Map(\Div^1_E, */\G_m^\an)
\end{equation*}
of mapping anima. In particular, any line bundle on $\Div^1_E$ is pulled back from $\Pic^1_E$.
\end{thm}

In particular, since isomorphism classes of line bundles on $\Div_E^1$ are naturally in bijection with characters of the Weil group $W_E$ of $E$ by \cite[Thm.\ VII.0.1]{RealLLC} and $\Pic^1_E$ is the classifying stack of $E^{\times, \la}$, we deduce local class field theory for archimedean local fields in the following form:

\begin{cor}
\label{cor:intro-cft}
For any archimedean local field $E$, there is a canonical isomorphism $W_E^\ab\cong E^\times$.
\end{cor}

We note that \cref{thm:intro-main} is an archimedean analogue of Fargues' result from \cite{Fargues} that rank one $\ell$-adic étale local systems on the stack $\Div^{1, \FF}$ of degree $1$ divisors on the Fargues--Fontaine curve are pulled back along the Abel--Jacobi map. His proof strategy is the following: Given a rank one local system $\mathcal{L}$ on $\Div^{1, \FF}$, its $d$-th symmetric power $\mathcal{L}^{\boxtimes{d}}$ yields a line bundle on $\Sym^d\Div^{1, \FF}\cong \Div^{d, \FF}$. Since the fibres of the $d$-th Abel--Jacobi map $\AJ^{d, \FF}: \Div^{d, \FF}\rightarrow\Pic^{d, \FF}$ are simply connected for $d\geq 2$, one can then descend $\mathcal{L}^{\boxtimes d}$ to $\Pic^{d, \FF}\cong\Pic^{1, \FF}$.

In our context, there are two reasons this strategy fails.
The first of these is that $\Div^d$ is not isomorphic to the $d$-th symmetric power $\Sym^d\Div^1$ of $\Div^1$, see \cref{exx:failure}.
The second is that, since we are working with quasi-coherent instead of $\ell$-adic sheaves, to obtain descent of line bundles along the Abel--Jacobi map, we would need the fibres to be ``contractible'' instead of just simply connected; however, the map $\AJ^d$ is only ``$(d-1)$-connective''. To address the first issue, we use the following simple idea motivated from algebraic geometry: For any $d$, we single out an open locus $\Div^d_{\HT\leq 1}$ in $\Div^d$ whose complement has ``codimension two'' and over which the natural map $\Sym^d \Div^1 \to \Div^d$ is an isomorphism; then we show that vector bundles extend uniquely from $\Div^d_{\HT\leq 1}$ to all of $\Div^d$. Meanwhile, the solution to the second issue is to pass to an Abel--Jacobi map ``at infinite level'': we assemble all the degree $d$ Abel--Jacobi maps into a single one, which will then have contractible fibers.

We caution the reader, however, that this discussion should only be regarded as an intuition: Indeed, since we are working with analytic stacks throughout the paper, we have no robust notion of dimension, fundamental group or codimension and hence none of the arguments sketched above are purely formal.

We conclude with a brief overview of the paper. In Section 2, we briefly recall some necessary background from \cite{RealLLC}. Afterwards, we identify finite-dimensional representations of a real Lie group $G$ with vector bundles on the classifying stack $*/G^\la$ in Section 3. Subsequently, in Section 4, we discuss the failure of $\Div^d$ to coincide with $\Sym^d \Div^1$, introduce the locus $\Div^d_{\HT\leq 1}$ and prove unique extension of vector bundles from $\Div^d_{\HT\leq 1}$ to $\Div^d$. In Section 5, we finally put all of this together in order to deduce \cref{thm:main} by passing to an Abel--Jacobi map ``at infinite level'' and descending from the case $E=\C$.

\bigskip

\noindent\textbf{Notation and conventions.}
 We work in the light setting. We will always work with the gaseous analytic ring structure; in particular, the $!$-topology coincides with the descendable topology in our setup. All our analytic stacks are over $\mathbb{C}_{\text{gas}}$.

\bigskip

\noindent\textbf{Acknowledgements.} We are grateful to Peter Scholze and Juan Esteban Rodríguez Camargo for several helpful discussions during the development of this project. Moreover, we thank Ferdinand Wagner for suggesting the proof of \cref{lem:symdcompacthausdorff}. The first author was part of the DFG RTG 2553. This work was carried out while the second author was a PhD student at the Max Planck Institute for Mathematics in Bonn and he would like to thank the institute for its hospitality. We also thank the Arizona Department of Transportation for freeing us from a snowstorm, during which this collaboration started.

\section{Recollections from \cite{RealLLC}}

We quickly recall the format of Scholze's geometric approach to the real local Langlands correspondence in \cite{RealLLC} and some of the basic statements that will be relevant later. The starting point is to replace the Fargues--Fontaine curve, which plays a central role in the Fargues--Scholze approach to the $\ell$-adic local Langlands correspondence from \cite{FarguesScholze}, by the following object: 
\begin{defi}
The \emph{twistor}-$\P^1$ is the projective $\R$-scheme
\begin{equation*}
    X_\R\coloneqq\widetilde{\P}^1_\R\coloneqq \Proj\R[x, y, z]/(x^2+y^2+z^2)\;.
\end{equation*}
Equivalently, it can be obtained by descending $\P^1_\C$ from $\C$ to $\R$ via the involution $z\mapsto -\ol{z}^{-1}$.
\end{defi}

The points $0$ and $\infty$ of $\P^1_\C$ descend to a distinguished $\C$-valued point of $X_\R$ that will be denoted by $\infty$. Moreover, from the presentation of $X_\R$ as $\Proj$ of a graded ring above, we obtain line bundles $\O(n)$ on $X_\R$ for each $n\in\Z$; beware that the pullback of $\O_{X_\R}(n)$ to $\P^1_\C$ coincides with $\O_{\P^1_\C}(2n)$, not $\O_{\P^1_\C}(n)$! To get the geometrisation machine running, one has to make this definition in families. These families will be parametrised by the following algebras (we refer the reader to \cite[§V.2]{RealLLC} for the definitions of bounded $\C$-algebras and their $\dagger$-reductions):

\begin{defi}
A bounded gaseous animated $\C$-algebra $A$ is called \emph{totally disconnected} if for each $s\in\pi_0\Spec A(*)$, the algebra
\begin{equation*}
    A_s\coloneqq \colim_{U\ni s} A(U)
\end{equation*}
satisfies $A_s^{\dagger\text{-red}}\cong\C$, where the colimit runs over all clopen neighbourhoods $U$ of $s$.
\end{defi}

For any such totally disconnected $\C$-algebra $A$, one gets a profinite set $S=\pi_0\Spec A(*)=\Hom(A, \C)$ and a map $A\rightarrow \Cont(S, \C)$ with kernel $\Nil^\dagger(A)$.

\begin{defi}
A totally disconnected $\C$-algebra $A$ is called \emph{strongly totally disconnected} if the map $A\rightarrow \Cont(S, \C)$ is surjective or, equivalently, $A^{\dagger\text{-red}}=\Cont(S, \C)$. We will often write $\ol{A}\coloneqq A^{\dagger\text{-red}}$.
\end{defi}

In the following, we will always restrict to countably presented totally disconnected algebras and hence $S$ will always be light. Then the most important fact about strongly totally disconnected $\C$-algebras is that they cover all (countably presented) totally disconnected $\C$-algebras in the descendable topology, see \cite[Prop.\ V.2.8, Rem.\ V.3.2]{RealLLC}.

\begin{defi} 
\label{def:pushpres}
For a totally disconnected $\C$-algebra $A$, the family $X_{\R, A}$ of twistor-$\P^1$'s over $A$ is defined as the pushout
\begin{equation*}
\begin{tikzcd}
    \AnSpec\Cont(S, \C)\ar[r, "\infty"]\ar[d] & X_\R\times_{\AnSpec\R} \AnSpec\Cont(S, \R)\ar[d] \\
    \AnSpec A\ar[r] & X_{\R, A}\nospacepunct{\;,}
\end{tikzcd}
\end{equation*}
where $S=\Hom(A, \C)$ and the top map is the base change of $\AnSpec\C\xrightarrow{\infty} X_\R$ to $\Cont(S, \R)$ over $\R$.
\end{defi}

Note that any line bundle $\O_{X_\R}(n)$ gives rise to a line bundle on $X_{\R, A}$ by gluing the pullback of $\O_{X_\R}(n)$ to $X_\R\times_{\AnSpec\R} \AnSpec\Cont(S, \R)$ with the structure sheaf on $\AnSpec A$ along $\AnSpec\Cont(S, \C)$. By abuse of notation, we will denote this line bundle also by $\O_{X_\R}(n)$.

Since we also want to treat the case $E=\C$ below, let us emphasise that there is a similar story for $\C$ in place of $\R$ and this can be obtained by just base changing the above constructions from $\R$ to $\C$. In particular, we have
\begin{equation*}
    X_\C\coloneqq X_\R\times_{\AnSpec\R}\AnSpec \C\cong \P^1_\C
\end{equation*}
and, for any totally disconnected $\C$-algebra $A$, the family $X_{\C, A}\coloneqq X_{\R, A}\times_{\AnSpec \R}\AnSpec\C$ can be described as the pushout
\begin{equation*}
\begin{tikzcd}
    \AnSpec\Cont(S, \C)\sqcup\AnSpec\Cont(S, \C)\ar[r, "0\sqcup\infty"]\ar[d] & \P^1_\C\times_{\AnSpec\C} \AnSpec\Cont(S, \C)\ar[d] \\
    \AnSpec(A\tensor_{\C, z\mapsto \ol{z}} \C)\sqcup\AnSpec A\ar[r] & X_{\C, A}\nospacepunct{\;.}
\end{tikzcd}
\end{equation*}

As in the case $E=\R$, any line bundle $\O_{\P^1_\C}(n)$ gives rise to a corresponding line bundle on $X_{\C, A}$. In the following, we will always use the descendable topology on the category $\TotDisc$ of (countably presented) totally disconnected $\C$-algebras and denote the category of stacks on this site by $\TDStack$, see \cite[Def.\ V.3.1]{RealLLC} for a precise definition. Note that any object of $\TDStack$ will give rise to an analytic stack over $\C$ by left Kan extension.

\begin{prop}
\label{prop:pic}
The stack $\Pic_E$ on $\TotDisc$ sending any totally disconnected $\C$-algebra $A$ to the groupoid of line bundles on $X_{E, A}$ is given by 
\begin{equation*}
    \Pic_E\cong \bigsqcup_{n\in\Z} */E^{\times, \la}
\end{equation*}
and this isomorphism is induced by the line bundles $\O_{X_E}(n)$. In particular, there is a well-defined degree map
\begin{equation*}
    \deg: \Pic_E\rightarrow\Z
\end{equation*}
sending $\O_{X_E}(n)$ to $n$ and we denote the preimage of $d\in\Z$ under this map by $\Pic_E^d$.
\end{prop}
\begin{proof}
See \cite[Thm.\ VI.2.4, Prop.\ VI.4.1]{RealLLC}.
\end{proof}

\begin{defi}
For each $n\in\Z$, the \emph{Banach--Colmez space} of $\O_{X_E}(n)$ is the stack $\BC_E(\O(n))$ on $\TotDisc$ sending any totally disconnected $\C$-algebra $A$ to the animated abelian group $\Gamma(X_{E, A}, \O_{X_E}(n))$.
\end{defi}

\begin{defi}
The stack $\Div^1_E$ of \emph{degree $1$ divisors} on $X_E$ is defined by sending any $A\in\TotDisc$ to the anima of pairs $(\L, s)$ of a degree $1$ line bundle $\L$ on $X_{E, A}$ and a global section $s$ of $\L$ which is nonzero after pullback along every map $A\rightarrow\C$.
\end{defi}

\begin{prop}
There is an isomorphism
\begin{equation*}
    \Div_E^1\cong (\BC_E(\O(1))\setminus \{0\})/E^{\times, \la}\;.
\end{equation*}
\end{prop}
\begin{proof}
See \cite[Prop.\ VI.3.2]{RealLLC}.
\end{proof}

Next, we need to recall the construction of Betti stacks associated to locally compact Hausdorff spaces $S$. This is constructed in two steps: First, there is a functor 
$$ \text{ProFin}^{\text{light}} \to \text{AnStack}_{\mathbb{C}}$$
sending $S$ to $\AnSpec(\Cont(S,\mathbb{C)})$, where we use the gaseous analytic ring structure. Since this sends hypercovers to $!$-hypercovers, there is a unique colimit-preserving extension to a functor
\begin{equation*}
\begin{split}
    \text{Cond}\Ani^{\text{light}}&\to \text{AnStack}_{\mathbb{C}} \\
    S&\mapsto S^{\Betti}
\end{split}
\end{equation*}
Restricting to totally disconnected algebras, the analytic stack $S^{\Betti}$ has an easy functor of points description on strongly totally disconnected algebras: Namely, given a totally disconnected algebra $A$ such that $\ol{A}=\Cont(T,\mathbb{C})$, we have
\begin{equation*}
    S^{\Betti}(A)=S^{\Betti}(\ol{A})=\Cont(T,S)\;.
\end{equation*}
The category of sheaves on $S^\Betti$ has the following easy description:

\begin{prop}\label{BasicBetti}
    Let $X$ be an analytic stack and $S$ a locally compact Hausdorff space which is locally of finite cohomological dimension. Then there is a natural isomorphism $$\D(X\times S^{\Betti})\cong \widehat{\D}(S,\D(X))$$
    of symmetric stable presentable monoidal infinity categories, where the right-hand side denotes the left-completion of the category of $\D(X)$-valued sheaves on $S$.
\end{prop}
\begin{proof}
See \cite[Cor.\ II.1.2]{RealLLC}
\end{proof}

Instead of the full derived category of $X\times S^{\Betti}$, we will often be interested in just the invertible objects, i.e.\ the Picard groupoid $\cal{P}\mathrm{ic}(X\times S^{\Betti})$. From \cref{BasicBetti}, we can deduce the following description:

\begin{lem}
\label{lem:picbetti}
Let $S$ be a locally compact Hausdorff space and $\mathcal{C}$ a compactly generated symmetric monoidal presentable $\infty$-category. Then the dualisable (resp.\ invertible) objects in $\D(S, \mathcal{C})$ are exactly the locally constant sheaves in the sense of \cite[A.1.12]{Highalg} with dualisable (resp.\ invertible) stalks.

In particular, if $S$ is any CW complex, we have 
\begin{equation*}
    \cal{P}\mathrm{ic}(X\times S^{\Betti})\cong \Map(|S|,\cal{P}\mathrm{ic}(X))
\end{equation*}
for any analytic stack $X$, where $|S|$ denotes the underlying anima of $S$.
\end{lem}
\begin{proof}
    If $\mathcal{F}$ in $\D(S,\mathcal{C})$ is locally constant with dualisable stalks, then $\mathcal{F}$ is constant over an open cover of $S$ with dualisable spaces of sections. Therefore, the natural map
    \begin{equation*}
        \mathcal{H}\mathrm{om}(\mathcal{F},1)\otimes \mathcal{F}\rightarrow\mathcal{H}\mathrm{om}(\mathcal{F},\mathcal{F})
    \end{equation*}
    is locally an isomorphism, hence globally and this implies dualisability of $\mathcal{F}$. 
    For the converse direction, first observe that taking stalks is symmetric monoidal, hence the stalks of $\mathcal{F}$ are dualisable. To prove local constancy, we can reduce to the case of compact Hausdorff spaces, in which case the unit object of $\D(S, \cal{C})$ is compact. This implies that any dualisable object is compact, hence dualisable sheaves are locally constant by \cite[Prop 6.15]{efimov2025ktheorylocalizinginvariantslarge}.

    For the last assertion, first note that we can reduce to the case of a finite CW complex since both sides of the claimed equivalence send colimits in $S$ to limits. Moreover, by descent in $X$, we can reduce to $X=\AnSpec A$, in which case $\D(X)=\D(A)$ is compactly generated. Now we can apply \cref{BasicBetti} to deduce that there is an equivalence 
    \begin{equation*}
        \cal{P}\mathrm{ic}(\D(X\times S^{\Betti}))\cong\cal{P}\mathrm{ic}(\D(S,\D(X)))
    \end{equation*}
    and, by the first part, the right-hand side is the category of locally constant sheaves on $S$ valued in $\cal{P}\mathrm{ic}(X)$. 
    But the groupoid of $\cal{C}$-valued locally constant sheaves on $S$ is equivalent to $\Map(|S|,\mathcal{C})$ by \cite[Appendix A]{Highalg}. Indeed, in loc.\ cit., the statement is proved for $\mathcal{C}=\Ani$ and then the general case follows by tensoring.
\end{proof}

Using Betti stacks, one can formulate the following version of the Riemann--Hilbert correspondence, which is \cite[Thm.\ II.3.1]{RealLLC}.

\begin{thm}
\label{anriemannhilb}
   Let $X$ be a complex manifold. The map $X\to X^{\Betti}$ factors uniquely over $X^{\an, \dR}\coloneqq X/\Delta(X)^{\dagger}$ and induces an isomorphism
   \begin{equation*}
       X^{\an, \dR}\cong X^{\Betti}\;.
   \end{equation*}
   Here, we use $\Delta(X)^\dagger\subseteq X\times X$ to denote the overconvergent neighbourhood of the diagonal.
\end{thm}

We will mostly apply the above result in the case $X=\G_m^\an$, where it asserts that there is an isomorphism $\G_m^\an/\G_m^\dagger\cong \C^{\times, \Betti}$, where $\G_m^\dagger$ denotes the overconvergent neighbourhood of $1\in\G_m^\an$. Finally, we will use the following lemma, which is implicitly used in \cite{RealLLC}:

\begin{lem}\label{lapullback}
    Let $X$ be a smooth algebraic variety over the real numbers. Then there is a natural pullback diagram 
    $$
    \begin{tikzcd}
    X(\R)^{\la} \arrow[d] \arrow[r] & X(\R)^{\Betti} \arrow[d] \\
    X(\C)^{\an} \arrow[r]            &  X(\C)^{\Betti} \nospacepunct{\;.}        
    \end{tikzcd}
    $$
\end{lem}
\begin{proof}
    Since the existence of such a commutative diagram is clear, we only need to check that it is cartesian, which we can do locally on $X(\C)^{\Betti}$ and, in particular, Zariski-locally on $X$. Then we may assume that $X$ admits an étale map to $\mathbb{A}^{n}$ and obtain a commutative cube
    $$\begin{tikzcd}
                                  & X(\R)^{\la} \arrow[rr] \arrow[ld] \arrow[dd] &                           & X(\R)^{\Betti} \arrow[dd] \arrow[ld] \\
X(\C)^{\an} \arrow[rr] \arrow[dd] &                                              & X(\C)^{\Betti} \arrow[dd] &                                      \\
                                  & {\mathbb{R}^{n,\la}} \arrow[ld] \arrow[rr]           &                           & {\R^{n,\Betti}} \arrow[ld]           \\
{\mathbb{A}^{n,\an}} \arrow[rr]           &                                              & {\C^{n,\Betti}}\nospacepunct{\;,}          &                                     
\end{tikzcd}$$
where the vertical maps are induced by the given étale map $X\rightarrow\A^n$. Then our task reduces to showing that each face but the top one is cartesian.

For each of the lateral faces, we just note that they are all induced by the map $X(\C)^\an\rightarrow \A^{n, \an}$, which locally on $\A^{n, \an}$ is given by a finite number of disjoint copies of the base as $X\rightarrow\A^n$ is étale. Finally, proving that the bottom face is cartesian reduces to the case $n=1$ by compatibility with finite products, where we note that both $\A^{1, \an}\rightarrow \C^\Betti$ and $\R^\la\rightarrow \R^\Betti$ are $\G_a^\dagger$-torsors by \cref{anriemannhilb}.
\end{proof}

\begin{cor}
\label{exact sequences}
The following is an exact sequence of totally disconnected stacks valued in animated abelian groups:
$$
\begin{tikzcd}
0 \arrow[r] & \mathbb{G}_{m}^{\dagger,2} \arrow[r] & \mathbb{C}^{\times,\la} \arrow[r] & {\mathbb{C}^{\times,\Betti}} \arrow[r] & 0\nospacepunct{\;.}
\end{tikzcd}
$$
In particular, we have
\begin{equation*}
    \mathbb{G}_{m}^{\an}/\mathbb{C}^{\times,\la}\cong*/\mathbb{G}_{m}^{\dagger}\;,\hspace{0.5cm} \mathbb{G}_{m}^{\an,2}/\mathbb{C}^{\times,\la}\cong \mathbb{C}^{\times, \Betti}\;.
\end{equation*}
\end{cor}
\begin{proof}
    The first part follows formally from \cref{anriemannhilb} and \cref{lapullback}. The second part follows from the first using \cref{anriemannhilb} and the nine lemma.
\end{proof}
\section{Finite-dimensional representations of real Lie groups}

Let $G$ be a real Lie group. Our goal in this section will be to show that vector bundles on the analytic stack $*/G^\la$ identify with finite-dimensional continuous representations of $G$.

\begin{lem}
The fully faithful embedding
\begin{equation*}
    \D(*/G^\la)\hookrightarrow \D(D_c(G))
\end{equation*}
from \cite[Prop.\ III.2.2]{RealLLC} restricts to an equivalence between the categories of perfect complexes on $\D(*/G^\la)$ and $D_c(G)$-modules in $\Perf(\C)$. Here, $D_c(G)$ is the algebra of compactly supported distributions on $G$ from \cite[Def.\ III.2.1]{RealLLC}.
\end{lem}
\begin{proof}
By \cite[Prop.\ III.2.2]{RealLLC}, containment in the essential image is equivalent to containment in the essential image of the fully faithful embedding
\begin{equation*}
    \D(*/(1\subseteq G^\la)^\dagger)\hookrightarrow \D(D(1\subseteq G))
\end{equation*}
from \cite[Prop.\ III.1.7]{RealLLC}, where $D(1\subseteq G)=(\O(1\subseteq G^\la)^\dagger)^*$ is the $\C$-linear dual of the space of overconvergent analytic functions at $1$ by definition. By the proof of \cite[Prop.\ III.1.7]{RealLLC}, this is in turn equivalent to containment in the essential image of the fully faithful embedding 
\begin{equation*}
    \D(*/(1\subseteq G^\la)^\dagger)\hookrightarrow \D(*/(1\subseteq G^\la)^\wedge)
\end{equation*}
from \cite[Prop.\ III.1.5]{RealLLC}, which may be checked on $1$-parameter subgroups of $G$ by loc.\ cit. 

We are thus reduced to checking that the fully faithful embedding
\begin{equation*}
    \D(*/\G_a^\dagger)\hookrightarrow\D(*/\widehat{\G}_a)\cong \D(\C_\gas[U])
\end{equation*}
from \cite[Prop.\ II.2.3]{RealLLC} induces an equivalence on categories of perfect complexes. By loc.\ cit., the essential image of this embedding identifies with those $\C_\gas[U]$-modules killed by the idempotent algebra
\begin{equation*}
A\coloneqq \left\lbrace \sum_{n\in\Z} a_nU^{-n}\in \C(\!(U^{-1})\!): |a_n|\frac{r^n}{n!}\rightarrow 0\text{ for some $r>0$}\right\rbrace\;,
\end{equation*}
so our task is to check that this kills any perfect complex $M$ over $\C$ equipped with an endomorphism $U: M\rightarrow M$. By induction on the cohomological degrees in which $M$ is nonzero, this reduces to the case where $M\cong \C^{\oplus n}$ is a finite-dimensional $\C$-vector space. However, then $U$ is given by an $n\times n$-matrix over $\C$ and by passing to direct summands of $M$ we can assume that $U$ only has a single Jordan block of eigenvalue $\lambda$. Tensoring with $\C[U, U^{-1}]$, which maps to $A$, already kills $M$ if $\lambda=0$, so we may assume that $\lambda\in\C^\times$. Then $1-\lambda^{-1}U^{-1}$ is nilpotent on $M$, but $A$ contains
\begin{equation*}
    \frac{1}{1-\lambda^{-1}U^{-1}}=\sum_{n\geq 0} \lambda^{-n}U^{-n}\;,
\end{equation*}
hence $M\tensor_{\C_\gas[U]} A=0$, as desired.

\end{proof}

\begin{lem}
The functor
\begin{equation*}
    \D(D_c(G))\rightarrow \D(\C[G]_\gas)
\end{equation*}
obtained by restricting scalars along $\C[G]_\gas\rightarrow D_c(G)$ induces an equivalence between the categories of $D_c(G)$-modules and $\C[G]_\gas$-modules in $\Perf(\C)$. Moreover, these categories are equivalent to the bounded derived category of finite-dimensional continuous $G$-representations.
\end{lem}
\begin{proof}
Since the two categories in question are equivalent to the bounded derived categories of finite-dimensional $D_c(G)$-modules and finite-dimensional $\C[G]_\gas$-modules, respectively, we can reduce to showing that any finite-dimensional $\C[G]_\gas$-module admits a unique compatible $D_c(G)$-module structure and that any $\C[G]_\gas$-linear map becomes automatically $D_c(G)$-linear via this construction.

To this end, note that finite-dimensional $\C[G]_\gas$-modules are the same as finite-dimensional continuous $G$-representations by the adjunction between $\C[-]_\gas$ and the forgetful functor from gaseous $\C$-vector spaces to condensed sets (this also proves the second part of the lemma). Given such a finite-dimensional continuous $G$-representation $V$, it is a standard fact in the representation theory of real Lie groups that its matrix coefficients are automatically analytic functions, i.e.\ that the canonical map $V\tensor V^*\rightarrow\Cont(G, \C)$ factors through a map $V\tensor V^*\rightarrow \O(G^\la)$. 

We claim that the map $V\tensor V^*\rightarrow \O(G^\la)$ induces the desired $D_c(G)$-module structure on $V$. Indeed, the given map is equivalent to a family of compatible maps $V\tensor V^*\rightarrow \O(Z\subseteq G^\la)^\dagger$ for $Z\subseteq G$ compact Stein, which dualise to compatible maps $D(Z\subseteq G)\rightarrow V^*\tensor V$. In turn, these induce a map $D_c(G)=\colim_{Z\subseteq G} D(Z\subseteq G)\rightarrow V^*\tensor V$, which is adjoint to a map $D_c(G)\tensor V\rightarrow V$. One easily checks that this yields a $D_c(G)$-module structure on $V$ extending the given $\C_\gas[G]$-module structure. Finally, uniqueness of this extension as well as the fact that any $\C_\gas[G]$-linear map automatically becomes $D_c(G)$-linear this way follows from the fact that $V$, being a finite-dimensional $\C$-vector space, is quasiseparated and $\C_\gas[G]$ is dense in $D_c(G)$.

\end{proof}

\begin{cor}\label{Findimrep}
Let $G$ be a real Lie group. Then the functor
\begin{equation*}
    \D(*/G^\la)\rightarrow \D(\C[G]_\gas)
\end{equation*}
from \cite[Question III.2.3]{RealLLC} induces an equivalence between the categories of perfect complexes on $*/G^\la$ and $\C[G]_\gas$-modules in $\Perf(\C)$. Moreover, these categories are equivalent to the bounded derived category of finite-dimensional continuous $G$-representations.
\end{cor}
\begin{proof}
The functor in question factors as
\begin{equation*}
    \D(*/G^\la)\rightarrow \D(D_c(G))\rightarrow \D(\C[G]_\gas)
\end{equation*}
and hence the statement follows from the two previous lemmas.
\end{proof}

\section{Line bundles on $\Div^1$ and $\Div^d$}

Our aim in this section is to understand the relation between $\Div^1$ and $\Div^d$ for $d\geq 2$, which is defined as follows: 

\begin{defi}
    For $d\geq 1$, the stack $\Div_E^d$ of \emph{degree $d$ divisors} on $X_E$ is defined by sending a totally disconnected algebra $A$ to the anima of pairs $(\L, s)$ of a degree $d$ line bundle $\L$ on $X_{E,A}$ and a global section $s$ of $\L$ which is nonzero after pullback along every map $A\rightarrow\C$.
\end{defi}

To explain more precisely what we want to do, recall from \cite[Prop.\ II.3.6]{FarguesScholze} that, in the setting of the geometrisation of the $\ell$-adic local Langlands correspondence of Fargues and Scholze, the natural product map sending a $d$-tuple $((\cal{L}_1, s_1), \dots, (\cal{L}_d, s_d))$ to the pair $(\cal{L}_1\tensor\dots\tensor\cal{L}_d, s_1\cdots s_d)$ induces an isomorphism
\begin{equation}
\label{eq:geomdivd-symddivd}
    \Sym^d \Div^{1, \FF}\cong \Div^{d, \FF}
\end{equation}
of diamonds, where the major caveat is that the left-hand side in the above isomorphism should be understood as the \emph{sheaf} quotient of $(\Div^{1, \FF})^d$ by the canonical permutation action of $S_d$ as opposed to the stacky quotient. This isomorphism enables one to produce a line bundle $\cal{L}^{(d)}$ on $\Div^{d, \FF}$ from any line bundle $\cal{L}$ on $\Div^{1, \FF}$ roughly by taking a $d$-fold exterior tensor product.

However, the isomorphism (\ref{eq:geomdivd-symddivd}) \emph{fails} in our setting. Namely, while the natural map
\begin{equation}
\label{eq:geomdivd-mapsymdtodivd}
    \Sym^d \Div^1\rightarrow \Div^d
\end{equation}
is still surjective as we will show below, it is not an isomorphism anymore. To salvage this, we will introduce an open substack $\Div^d_{\HT\leq 1}$ of $\Div^d$ roughly parametrising degree $d$ divisors whose order of vanishing at $\infty$ is at most $1$. Then we will prove that the map (\ref{eq:geomdivd-mapsymdtodivd}) is an isomorphism after pullback to $\Div^d_{\HT\leq 1}$. Indeed, this will be enough for our purposes: As a second step, we will prove that 
\begin{equation*}
    \Map(\Div^d_{\HT\leq 1}, */\G_m^\an)\cong \Map(\Div^d, */\G_m^\an)
\end{equation*}
via the natural map, i.e.\ line bundles extend uniquely from $\Div^d_{\HT\leq 1}$ to $\Div^d$. While, intuitively speaking, this is due to the fact that $\Div^d_{\HT\leq 1}$ is an open substack of $\Div^d$ whose closed complement has ``codimension $2$'', we will instead prove this by an analysis of the stack $\Div^d$ in terms of charts which are explicit quotients. 

\subsection{The geometry of $\Div^d$}

Before we can begin, we first recall some basic results about $\Div^d$, starting with the following presentation: 

\begin{prop} 
\label{prop:divd}
    There is an isomorphism
    $$\Div_E^d\cong(\BC_E(\O(d))\setminus \{0\})/E^{\times,\la}\;.$$
\end{prop}
\begin{proof}
    Given any pair $(\L,s)$, the line bundle $\L$ is locally isomorphic to $\mathcal{O}(d)$ and, under this isomorphism, $s$ is sent to a fibrewise non-vanishing section, see also \cite[Prop.s VI.3.2, VI.4.3]{RealLLC}.
\end{proof}

The Banach-Colmez spaces appearing in the presentation above are fairly easy to describe:

\begin{prop}
\label{prop:bcspaces}
    Let $d\geq 1$. For $E=\mathbb{R}$, we have
    $$\BC_E(\O(d))\cong \A^{1, \an}\times_{\C^{\Betti}} H^{0}(X_{E},\mathcal{O}(d))^{\Betti} \cong \A^{1, \an}\times \R^{2d-1, \Betti}\;.$$
    For $E=\mathbb{C}$, we have
    $$\BC_E(\O(d))\cong \A^{2,\an}  \times_{(\C\times \C)^{\Betti}} 
H^{0}(X_{E},\O(d))^{\Betti}\cong \A^{2,\an}  \times \C^{d-1, \Betti}\;.
$$
\end{prop}
\begin{proof}
The first isomorphisms are clear from the pushout presentation of $X_{E, A}$ while the second ones follow by using $H^0(X_\R, \O(d))\cong \R^{2d+1}$ and $H^0(X_\C, \O(d))\cong \C^{d+1}$, which can both easily be calculated from the $\Proj$-description of $X_E$.
\end{proof}

From now on, we will restrict to the case $E=\C$, which will suffice for our applications. Here, we will later on also need a description of $\Div^d$ by charts. To this end, let us first introduce some notation: For any strongly totally disconnected $\C$-algebra $A$ with $S=\Hom(A, \C)$, an $A$-valued point of $H^0(\P^1_\C, \O(d))^\Betti$ will be a homogeneous degree $d$ polynomial $\sum_{i=0}^{d}c_i x^i y^{d-i}$, where each $c_i$ is a continuous map $S\rightarrow\C$. In this notation, the morphism
\begin{equation*}
    H^0(\P^1_\C,\O(d))^{\Betti}\rightarrow (\C\times \C)^{\Betti}
\end{equation*}
from \cref{prop:bcspaces} is given by evaluating at $0=[0:1]$ and $\infty=[1:0]$ on $\P^1_\C$.

\begin{lem}
\label{lem:charts}
The stack $\Div^d$ has an open cover given by
\begin{equation*}
    \text{$d-1$ copies of }\mathbb{A}^{2,\an}/(\mathbb{G}_{m}^\dagger)^2\times\C^{d-2, \Betti}\hspace{0.5cm}\text{and}\hspace{0.5cm}\text{$2$ copies of }\mathbb{A}^{1,\an}/\mathbb{G}_{m}^{\dagger}\times\C^{d-1, \Betti}\;.
\end{equation*}
\end{lem}
\begin{proof}
    By the description of \cref{prop:divd}, we need at least one of the $c_k$ to be nonzero and hence we can cover $\Div^d$ by open substacks $\{c_i\neq 0\}$. For $i=0$ or $i=d$, we get the open substack given by
\begin{equation*}
    (\mathbb{G}_{m}^{\an}\times\mathbb{A}^{1,\an}\times\C^{d-1, \Betti})\,/\,\C^{\times,\la}
    \cong\mathbb{A}^{1,\an}/\mathbb{G}_{m}^{\dagger}\times\C^{d-1, \Betti}\;,    
\end{equation*}
    where the isomorphism is given by using the first canonical isomorphism in \cref{exact sequences}. For $1\leq i\leq d-1 $, we instead obtain
\begin{equation*}
    \mathbb{A}^{2,\an}\times \C^{d-2,\Betti}\times \C^{\times,\Betti}\,/\,\C^{\times,\la}\cong\mathbb{A}^{2,\an}/(\mathbb{G}_{m}^\dagger)^2\times\C^{d-2, \Betti}\;,
\end{equation*}
    where the isomorphism is due to the exact sequence of \cref{exact sequences}.
\end{proof}

\subsection{$\Div^d$ versus $\Sym^d\Div^1$}

We first recall how to define symmetric powers of stacks.

\begin{defi}
The functor
\begin{equation*}
    \Sym^d: \Ani\rightarrow\Ani
\end{equation*}
is defined by animating the functor
\begin{equation*}
    \Sym^d: \Set\rightarrow\Set\;.
\end{equation*}
as described in \cite[§5.1.4]{CesnaviciusScholze}. For any sheaf of anima on a site, we can then define its $d$-th symmetric power by applying $\Sym^d$ pointwise and then sheafifying.
\end{defi}

\begin{rem}
Putting together all $\Sym^d$, we obtain a functor $\Sym: \Ani\rightarrow \Ani$ given by $X\mapsto \Sym X\coloneqq \bigsqcup_{d\geq 0} \Sym^d X$. Then $\Sym X$ is the free animated monoid on $X$. We note that this is \emph{not} the same as the free monoid object in anima on $X$: Indeed, the free animated monoid on $*$ is $\mathbb{N}$ while the free monoid object in anima on $*$ is $\bigsqcup_{d\geq 0} */S_d$, where $S_d$ is the symmetric group on $d$ elements.
\end{rem}

Our first goal is to show that the map $\Sym^d\Div^1\rightarrow\Div^d$ induces an isomorphism on analytic de Rham stacks, which we recall are defined as follows:

\begin{defi}
Let $X$ be an object of $\TDStack$. Then we define its \emph{analytic de Rham stack} as the stack $X^{\an, \dR}$ on totally disconnecteds given by
\begin{equation*}
    X^{\an, \dR}(A)\coloneqq X(\ol{A})\;.
\end{equation*}
\end{defi}

\begin{rem}
In the case where $X$ is a complex analytic manifold, one can check that this agrees with the definition from \cref{anriemannhilb}.
\end{rem}

Our first goal is thus to prove the following:

\begin{prop}
\label{prop:symdivdr}
On the level of analytic de Rham stacks we have the following isomorphism:
\begin{equation*}
    \Sym^d(\Div^1)^{\an, \dR}\cong (\Div^d)^{\an, \dR}\;.
\end{equation*}
\end{prop}

We start by establishing some preparatory lemmas.

\begin{lem}
\label{lem:symdcompacthausdorff}
Let $X$ be a metrisable compact Hausdorff space. Then
\begin{equation*}
    \Sym^d \ul{X}\cong \ul{\Sym^d X}\;,
\end{equation*}
where we define the topological space $\Sym^d X$ as the quotient of $X^d$ by the action of $S_d$ and $\ul{(-)}$ denotes the functor from topological spaces to (light) condensed sets.
\end{lem}
\begin{proof}
First note that there is a natural map
\begin{equation*}
    \Sym^d \ul{X}\rightarrow\ul{\Sym^d X}\;.
\end{equation*}
If we can show that $\Sym^d \ul{X}$ is qcqs, then we are done: In that case, both source and target of the above map are in the image of the full inclusion of metrisable compact Hausdorff spaces into light condensed sets. Moreover, the map is clearly a bijection on $*$-valued points (as sheafification does not change the value of a condensed set on the point) and then we are done by the elementary fact that bijections between compact Hausdorff spaces are homeomorphisms.

Now observe that $\ul{X}^d\rightarrow \Sym^d \ul{X}$ is a surjection and hence $\Sym^d \ul{X}$ is quasicompact since $\ul{X}^d$ is. Furthermore, by definition of $\Sym^d\ul{X}$, we have a pullback diagram in condensed sets
\begin{equation*}
\begin{tikzcd}
    \ul{\bigcup_{\sigma\in S_d} (\sigma\times \id)(\Delta(X^d))}\ar[r]\ar[d] & \ul{X^d\times X^d}\ar[d] \\
    \Sym^d \ul{X}\ar[r, "\Delta"] & \Sym^d \ul{X}\times\Sym^d \ul{X}\nospacepunct{\;,}
\end{tikzcd}
\end{equation*}
where $\Delta(X^d)$ denotes the image of the diagonal embedding of $X^d$ into $X^d\times X^d$. As all $(\sigma\times\id)(\Delta(X^d))$ are closed subsets of $X^d\times X^d$ due to $X$ being Hausdorff, we know that their union is a closed subset of $X^d\times X^d$ as well. Since $\ul{X}^d\times\ul{X}^d\cong \ul{X^d\times X^d}\rightarrow \Sym^d\ul{X}\times\Sym^d\ul{X}$ is a surjection, this shows that $\Delta$ is a quasicompact morphism and thus $\Sym^d\ul{X}$ is quasiseparated, so we win.
\end{proof}
\begin{cor}
\label{cor:symbetti}
For any metrisable compact Hausdorff space $X$, we have
\begin{equation*}
    \Sym^d(X^\Betti)\cong (\Sym^d X)^\Betti
\end{equation*}
as objects of $\TDStack$.
\end{cor}
\begin{proof}
For any totally disconnected $A$, we have $X^\Betti(A)=\Cont(\Hom(A, \C), X)$, where $\Hom(A, \C)$ is a profinite set, and similarly for $A$-valued points of $(\Sym^d X)^\Betti$. Now we are done by the previous lemma.
\end{proof}

We are now in a position to prove \cref{prop:symdivdr}.

\begin{proof}[Proof of \cref{prop:symdivdr}]
Note that, by \cref{prop:bcspaces} and the analytic Riemann--Hilbert isomorphism from \cref{anriemannhilb}, we have
\begin{equation*}
    \BC(\O(d))^{\an, \dR}\cong H^0(\P^1_\C, \O(d))^\Betti
\end{equation*}
and hence
\begin{equation*}
    (\Div^d)^{\an, \dR}\cong (H^0(\P^1_\C, \O(d))\setminus \{0\}/\C^\times)^\Betti\;,
\end{equation*}
where we have used that $(\C^{\times, \la})^{\an, \dR}\cong \C^{\times, \Betti}$. Thus, by \cref{cor:symbetti}, we are reduced to proving that
\begin{equation*}
    \Sym^d (H^0(\P^1_\C, \O(1))\setminus \{0\}\,/\,\C^\times)\cong H^0(\P^1_\C, \O(d))\setminus \{0\}\,/\,\C^\times
\end{equation*}
as topological spaces via the product map. However, this just amounts to the well-known fact that
\begin{equation*}
    \Sym^d \C\P^1\cong \C\P^d
\end{equation*}
via the map 
\begin{equation*}
    ([a_1:b_1], \dots, [a_d:b_d])\mapsto [c_1:\dots:c_{d+1}]\;,
\end{equation*}
where $c_i$ is the coefficient of $x^iy^{d-i}$ in $\prod_i (a_ix+b_iy)$. 
\end{proof}

Next, we prove that, even before taking de Rham stacks, the product map is at least surjective:

\begin{lem}
\label{lem:prodsurjects}
The product map 
\begin{equation*}
    (\Div^1)^d\rightarrow\Div^d
\end{equation*}
is a surjection.
\end{lem}
\begin{proof}
As surjectivity is a condition on $\pi_0$, we may reduce to $A$ being a static totally disconnected $\C$-algebra; indeed, note that $\pi_0\BC(\O(d))(A)$ only depends on $\pi_0(A)$. Let $S\coloneqq \Hom(A, \C)$ and take any $A$-point of $\Div^d$. After localising, we may assume that it comes from an $A$-point of $\BC(\O(d))\setminus \{0\}$, i.e.\ it corresponds to a tuple $(c_0, c_1, \dots, c_d)$ with $c_1, \dots, c_{d-1}\in \Hom(S, \C)$ and $c_0, c_d\in A$ by \cref{prop:bcspaces}. After localising further, we may assume that $c_1, \dots, c_{d-1}$ actually come from $\ol{A}$. Then our goal is to produce $d$ pairs $(a_i, b_i)\in A\times A$ such that
\begin{equation*}
    \prod_i (a_ix+b_iy)=\sum_i c_i x^i y^{d-i}\;.
\end{equation*}
More precisely, we need to ensure the above equality in $\ol{A}[x, y]$ together with $\prod_i a_i=c_d$ and $\prod_i b_i=c_0$ in $A$. As we already know that $(\Div^1)^d\rightarrow \Div^d$ induces a surjection on analytic de Rham stacks by \cref{prop:symdivdr}, we are then reduced to proving the following statement: Given an element $a\in A$, we can lift any factorisation of the image of $a$ in $\ol{A}$ into $d$ factors to $A$, at least locally on $A$. Applying this to $a=c_0$ and $a=c_d$ will then yield the claim.

Clearly, we are immediately reduced to the case $d=2$ by induction. Then we have $a_1, a_2\in A$ with $a_1a_2=a+r$ for some $r\in\Nil^\dagger(A)$ and we are looking for $t_1, t_2\in\Nil^\dagger(A)$ such that $(a_1+t_1)(a_2+t_2)=a$. As $A$ is bounded, we can find some $B\in\R$ such that $a_1, a_2$ and $a$ are each classified by a map from $\C[\widehat{B^{-1}z}]_\gas$, where $\C[\widehat{z}]_\gas$ denotes the free gaseous $\C$-algebra on a topologically nilpotent element. Then we are reduced to the universal case of showing that
\begin{equation*}
    A'\coloneqq A\{t_1, t_2\}^\dagger/((a_1+t_1)(a_2+t_2)=a)
\end{equation*}
is a descendable algebra over
\begin{equation*}
    A\coloneqq \C[\widehat{B^{-1} a_1}, \widehat{B^{-1} a_2}, \widehat{B^{-1} a}]_\gas\{r\}^\dagger/(a_1a_2=x+r)\;,
\end{equation*}
where $\C\{z\}^\dagger\coloneqq\colim_{\lambda>0} \C[\widehat{\lambda z}]_\gas$ is the free gaseous $\C$-algebra on a $\dagger$-nilpotent element. By \cite[Lem.\ 5.4.6]{FFdR}, it suffices to separately consider the cases where the maps $\AnSpec A\rightarrow\A^{1, \an}$ classifying $a_1$ and $a_2$ both factor over $\G_a^\dagger$ or where at least one of them factors over $\G_m^\an$. In the latter case, we can assume without loss of generality that the map factoring over $\G_m^\an$ is the one classifying $a_1$ and then the map $A\rightarrow A'$ has a section given by $t_1\mapsto 0$ and $t_2\mapsto -ra_1^{-1}$. 

In the remaining case $a_1, a_2\in\Nil^\dagger(A)$, observe that
\begin{equation*}
    A''\coloneqq A[t]/((a_1+t)(a_2+t)=a)
\end{equation*}
is a descendable $A$-algebra since it is just isomorphic to $A\oplus tA$ as an $A$-module and, clearly, $A$ is a retract of this. Moreover, the assumption that $a_1$ and $a_2$ are $\dagger$-nilpotent ensures that $t$ solving the equation
\begin{equation*}
    0=t^2+t(a_1+a_2)+(a_1a_2-a)=t^2+t(a_1+a_2)+r
\end{equation*}
is automatically $\dagger$-nilpotent as well. Thus, there is a map $A'\rightarrow A''$ given by $t_1, t_2\mapsto t$ and this concludes the proof as descendability satisfies cancellation.
\end{proof}

\begin{exx}\label{exx:failure}
The multiplication map $\Sym^d \Div^1\rightarrow\Div^d$ is not injective. To see this, take $A=\C\{t\}^\dagger$ and $d=2$. Then take the $A$-point of $\Div^2$ defined by the function $x^2\in\BC(\O(2))(A)$. Consider the $A$-points of $(\Div^1)^2$ obtained from the pairs $(x, x)$ and $(x, x+ty)$ of $A$-points of $\BC(\O(1))$. Then they both map to $x^2\in \Div^2(A)$ under the multiplication map: Indeed, note that $x\cdot x=x^2$ while $x(x+ty)=x^2+txy$, but the coefficient of $xy$ is only remembered in $\ol{A}$, where $t=0$. However, $(x, x)$ and $(x, x+ty)$ do not define the same $A$-point of $(\Div^1)^2$ up to permutation: Indeed, for this we would need that $t$ is mapped to zero under the multiplication action of $\C^{\times, \la}(A)$, which is impossible. 

Note that this issue can also not be solved by passing to a descendable cover of $A$. Indeed, for this we would need the map $\C\{t\}^\dagger\rightarrow\C$ given by $t\mapsto 0$ to be descendable, which is wrong: Then $\C\{t\}^\dagger\rightarrow\C[t]$ would be descendable as well, but $\C\{t\}^\dagger$ is an idempotent algebra over $\C[t]$.
\end{exx}

To what extent can we salvage the failure of the map $\Sym^d \Div^1\rightarrow\Div^d$ to be an isomorphism? Note that, by \ref{prop:symdivdr}, we already know that this is true after base changing to the locus in $\Div^d$ where the degree $d$ divisor neither vanishes at $0$ nor at $\infty$ since this open substack of $\Div^d$ is isomorphic to its own analytic de Rham stack (which in turn can be seen using \cref{lem:charts}). However, this locus is slightly too small for our purposes.

\begin{defi}
For any $d\geq 1$, we define $\Div^d_{\HT\leq 1}$ to be the stack sending a totally disconnected $\C$-algebra $A$ to the anima of pairs $(\cal{L}, s)$ of a degree $d$ line bundle $\cal{L}$ on $X_{\C, A}$ and a global section $s$ of $\cal{L}$ such that the order of vanishing of $s$ at $0$ and at $\infty$ is at most $1$ after pullback along every map $A\rightarrow\C$.
\end{defi}

\begin{rem}
The notation $\Div^d_{\HT\leq 1}$ is meant to evoke the fact that the preimage of $\Div^d_{\HT\leq 1}\subseteq\Div^d$ under the multiplication map are exactly those $d$-tuples of degree $1$ divisors for which at most one lies ``at the Hodge--Tate point'' of $\Div^1$.
\end{rem}

Note that $\Div^d_{\HT\leq 1}$ is an open substack of $\Div^d$ and that it is equal to all of $\Div^d$ if $d=1$. Moreover, being contained in $\Div^d_{\HT\leq 1}$ can be checked on the level of analytic de Rham stacks, i.e.\ the diagram
\begin{equation*}
\begin{tikzcd}
    \Div^d_{\HT\leq 1}\ar[r]\ar[d] & \Div^d\ar[d] \\
    (\Div^d_{\HT\leq 1})^{\an, \dR}\ar[r] & (\Div^d)^{\an, \dR}
\end{tikzcd}
\end{equation*}
is cartesian. Let us first observe the following analogue of \cref{lem:charts}:

\begin{lem}
\label{lem:chartshtleq1}
For $d\geq 3$, the stack $\Div^d_{\HT\leq 1}$ has an open cover given by
\begin{equation*}
\begin{split}
    \text{$2$ copies of }&\mathbb{A}^{1,\an}/\mathbb{G}_{m}^{\dagger}\times\C^{\times, \Betti}\times \C^{d-2, \Betti}\hspace{0.3cm}\text{,}\hspace{0.3cm}\text{$1$ copy of }\C^{\times, \Betti}\times \C^{d-1, \Betti} \\
    &\text{and $1$ copy of }\A^{2, \an}/(\G_m^\dagger)^2\times \C^{\times, \Betti}\times \C^{d-3, \Betti}\;.
\end{split}
\end{equation*}
If $d=2$, we have to replace the last piece of the cover above by
\begin{equation*}
    \A^{2, \an}/(\G_m^\dagger)^2\;.
\end{equation*}
\end{lem}
\begin{proof}
We use the presentation of $\Div^d$ from \cref{prop:divd} and the notation from the proof of \cref{lem:charts}. Then note that being contained in $\Div^d_{\HT\leq 1}$ amounts to the condition that at least one of $c_0, c_1$ and at least one of $c_d, c_{d-1}$ has to be a unit. Note that this automatically implies that $c_i$ is invertible for at least one $i$.

If $c_0, c_{d-1}\neq 0$, we obtain
\begin{equation*}
    (\G_m^\an\times \A^{1, \an}\times \C^{\times, \Betti}\times \C^{d-2, \Betti})\,/\,\C^{\times, \la}\cong \A^{1, \an}/\G_m^\dagger\times\C^{\times, \Betti}\times \C^{d-2, \Betti}
\end{equation*}
by the first canonical isomorphism of \cref{exact sequences} and similarly if $c_1, c_d\neq 0$. If $c_0, c_d\neq 0$, we get
\begin{equation*}
    (\G_m^{2, \an}\times \C^{d-1, \Betti})\,/\,\C^{\times, \la}\cong \C^{\times, \Betti}\times \C^{d-1, \Betti}
\end{equation*}
by the second canonical isomorphism of \cref{exact sequences}. Finally, if $c_1, c_{d-1}\neq 0$ and $d\geq 3$, we obtain
\begin{equation*}
    (\A^{2, \an}\times (\C^{\times, \Betti})^2\times \C^{d-3, \Betti})\,/\,\C^{\times, \la}\cong \A^{2, \an}/(\G_m^\dagger)^2\times \C^{\times, \Betti}\times \C^{d-3, \Betti}
\end{equation*}
using the exact sequence of \cref{exact sequences}. In the case $d=2$, the two conditions $c_1\neq 0$ and $c_{d-1}\neq 0$ are the same, in which case we instead get
\begin{equation*}
    (\A^{2, \an}\times \C^{\times, \Betti})\,/\,\C^{\times, \la}\cong \A^{2, \an}/(\G_m^\dagger)^2
\end{equation*}
using \cref{exact sequences} once again.
\end{proof}

Now we can state our desired result relating $\Div^1$ and $\Div^d$.

\begin{prop}
\label{prop:symdisohtleq1}
The map 
\begin{equation*}
    \Sym^d\Div^1\rightarrow\Div^d
\end{equation*}
induced by the multiplication map is an isomorphism over $\Div^d_{\HT\leq 1}\subseteq\Div^d$.
\end{prop}
\begin{proof}
By \cref{lem:prodsurjects}, we already know that the map is surjective so it remains to prove that it is also injective, i.e.\ we have to check that the diagonal map
\begin{equation}
\label{eq:geomdivd-diagonal}
    \Sym^d\Div^1\times_{\Div^d} \Div^d_{\HT\leq 1}\rightarrow (\Sym^d\Div^1\times_{\Div^d} \Sym^d\Div^1)\times_{\Div^d} \Div^d_{\HT\leq 1}
\end{equation}
is an isomorphism. Note that we already know this after taking de Rham stacks by \cref{prop:symdivdr} and hence, after possibly localising, we may fix an $\ol{A}$-point of the source, which we may assume to come from an $\ol{A}$-point $((\ol{a}_0, \ol{b}_0), \dots, (\ol{a}_d, \ol{b}_d))$ of $\BC(\O(1))^d$ after localising further. As we are working over $\Div^d_{\HT\leq 1}$, we may assume that all but one of the $\ol{a}_i$ and all but one of the $\ol{b}_i$ are invertible in $\ol{A}$.

Let $\ol{a}_k$ and $\ol{b}_\ell$ be the two possibly noninvertible elements. Note that, up to the multiplication action of $\G_m^\dagger$, all $\ol{a}_i, \ol{b}_i$ except $\ol{a}_k$ and $\ol{b}_\ell$ lift uniquely from $\ol{A}$ to $A$, and we will denote these lifts by $a_i, i\neq k$ and $b_i, i\neq\ell$. Then, by \cref{lem:charts}, up to the action of $\G_m^\dagger$, lifting the given $\ol{A}$-point of $\Sym^d\Div^1$ to an $A$-point of the right-hand side of (\ref{eq:geomdivd-diagonal}) amounts to lifting $a_k$ and $b_\ell$ from $\ol{A}$ to $A$ in possibly two different ways, say to $\tilde{a}_k, \tilde{b}_\ell$ and $\tilde{a}_k', \tilde{b}_\ell'$, and giving an identification of the two induced $A$-points of $\Div^d$. Such an identification in turn amounts to an identification
\begin{equation*}
    \widetilde{a}_k\prod_{i\neq k} a_i=\widetilde{a}'_k\prod_{i\neq k} a_i
\end{equation*}
up to the action of $\G_m^\dagger$ on both sides, and similarly for the $b_i$. However, by assumption, the $a_i, i\neq k$ are invertible and hence the above identification amounts to an identification of the lifts $\tilde{a}_k$ and $\tilde{a}_k'$ up to the action of $\G_m^\dagger$; analogously, the corresponding identification involving the $b_i$ translates into an identification of the lifts $\tilde{b}_\ell$ and $\tilde{b}_\ell'$ up to the action of $\G_m^\dagger$.

However, this precisely means that the diagonal map (\ref{eq:geomdivd-diagonal}) is an isomorphism over the $\ol{A}$-point of $\Sym^d\Div^1$ we fixed in the beginning, and this concludes the proof.
\end{proof}

\subsection{Extending line bundles from $\Div^d_{\HT\leq 1}$ to $\Div^d$}

As announced above, our next goal is to show that line bundles extend uniquely along $\Div^d_{\HT\leq 1}\subseteq \Div^d$. In fact, we will show a slight strengthening, but we still begin by proving the following assertion:

\begin{lem}
\label{lem:linebundlesextend}
Any line bundle on $\Div^d_{\HT\leq 1}$ extends uniquely to $\Div^d$ up to isomorphism.
\end{lem}
\begin{proof}
We first note that, by \cref{lem:charts} and \cref{lem:picbetti}, line bundles on $\Div^d$ identify with line bundles on $\A^{2, \an}/(\G_m^\dagger)^2$ such that the two monodromies obtained by pulling back along
\begin{equation*}
    \C^{\times, \Betti}\times\C^{\times, \Betti}\cong \G_m^{2, \an}/(\G_m^\dagger)^2\rightarrow \A^{2, \an}/(\G_m^\dagger)^2
\end{equation*}
are the same. Indeed, \cref{lem:picbetti} lets us ignore $\C^\Betti$-factors in the description of the charts and the compatibility of the monodromies is forced by the fact that the two copies of $\A^{1, \an}/\G_m^\dagger$ are glued along $\C^{\times, \Betti}$.

Thus, we have to show that line bundles on $\Div^d_{\HT\leq 1}$ admit the same description, for which we use \cref{lem:chartshtleq1} and \cref{lem:picbetti}. Taking any line bundle $\cal{L}$ on $\Div^d_{\HT\leq 1}$, the same argument as in the case of $\Div^d$ first shows that the two monodromies obtained from the two $\A^{1, \an}/\G_m^\dagger$-factors have to agree. Twisting by the pullback of a suitable line bundle along the inclusion $\Div^d_{\HT\leq 1}\rightarrow\Div^d$, we may thus assume that both of these monodromies are trivial. The proof of \cite[Prop.\ VII.2.2]{RealLLC} then shows that the restriction of $\cal{L}$ to the chart $\A^{2, \an}/(\G_m^\dagger)^2\times\C^{\times, \Betti}$ corresponds to a one-dimensional $\Z^2$-filtered $\C$-vector space together with an automorphism; looking at the other charts, we see that $\cal{L}$ is determined by this data and that the automorphism is trivialised after forgetting the filtration. However, then the automorphism must be trivial to begin with and hence $\cal{L}$ corresponds to a line bundle on $\A^{2, \an}/(\G_m^\dagger)^2$, as desired.
\end{proof}

Now we are in position to prove the following strengthening of \cref{lem:linebundlesextend}:

\begin{prop}
\label{prop:mappingspaces}
Pullback along the inclusion $\Div^d_{\HT\leq 1}\subseteq\Div^d$ induces an equivalence of Picard anima, i.e.\ we have
\begin{equation*}
    \Map(\Div^d, */\G_m^\an)\xrightarrow{\cong} \Map(\Div^d_{\HT\leq 1}, */\G_m^\an)
\end{equation*}
via the natural map.
\end{prop}
\begin{proof}
We first observe that both mapping anima are $1$-truncated. Indeed, for the left-hand side, it suffices to show that the Picard anima of each of the charts from \cref{lem:charts} is $1$-truncated as the inclusion of $1$-truncated anima into all anima is a left-adjoint and hence commutes with limits. By contractibility of $\C$ and \cref{lem:picbetti}, this then reduces to showing that $\Map(\A^{2, \an}/(\G_m^\dagger)^2, */\G_m^\an)$ and $\Map(\A^{1, \an}/\G_m^\dagger, */\G_m^\an)$ are $1$-truncated. We prove this for the second mapping space, the proof for the first one is similar. Then note that 
\begin{equation*}
    \A^{1, \an}/\G_m^\dagger\cong \colim_{\Delta} \A^{1, \an}\times (\G_m^\dagger)^\bullet
\end{equation*}
and hence we are reduced to showing that each $\Map(\A^{1, \an}\times (\G_m^\dagger)^\bullet, */\G_m^\an)$ is $1$-truncated by the same argument as before. Writing $\A^{1, \an}=\colim_N \ol{\mathbb{D}}(N)$ as a colimit of overconvergent disks, we may further reduce to checking that each $\Map(\ol{\mathbb{D}}(N)\times (\G_m^\dagger)^\bullet, */\G_m^\an)$ is $1$-truncated, but this is clear as $\ol{\mathbb{D}}(N)\times (\G_m^\dagger)^\bullet$ is an affine analytic stack whose ring of functions is static. Therefore, we conclude that $\Map(\Div^d, */\G_m^\an)$ is $1$-truncated and a similar argument using \cref{lem:chartshtleq1} in place of \cref{lem:charts} shows the same for $\Div^d_{\HT\leq 1}$ in place of $\Div^d$.

Thus, we are reduced to checking that line bundles extend uniquely from $\Div^d_{\HT\leq 1}$ to $\Div^d$ and that the same is true for their automorphisms. The first part follows from \cref{lem:linebundlesextend}, so take any line bundle $\cal{L}$ on $\Div^d$ and consider an automorphism of its pullback to $\Div^d_{\HT\leq 1}$. After twisting, we may reduce to $\cal{L}=\O$ and then it suffices to show that
\begin{equation*}
    H^0(\Div^d, \O)\cong H^0(\Div^d_{\HT\leq 1}, \O)
\end{equation*}
via the natural map. In turn, this follows once we can prove that the global sections of $\O$ on each of the charts from \cref{lem:charts} and \cref{lem:chartshtleq1} are given by $\C$. However, using \cref{lem:picbetti} and the fact that $\C^\times$ and $\C$ are both connected, this reduces to proving that
\begin{equation*}
    H^0(\A^{2, \an}/(\G_m^\dagger)^2, \O)\cong \C\cong H^0(\A^{1, \an}/\G_m^\dagger, \O)\;.
\end{equation*}
As $\A^{2, \an}/(\G_m^\dagger)^2\rightarrow \A^{2, \an}/\G_m^{2, \an}$ is a $(\C^{\times, \Betti})^2$-torsor while $\A^{1, \an}/\G_m^\dagger\rightarrow \A^{1, \an}/\G_m^\an$ is a $\C^{\times, \Betti}$-torsor, connectedness of $\C^\times$ and \cref{lem:picbetti} further reduce us to showing that
\begin{equation*}
    H^0(\A^{2, \an}/\G_m^{2, \an}, \O)\cong \C\cong H^0(\A^{1, \an}/\G_m^\an, \O)\;,
\end{equation*}
but this follows from the proof of \cite[Prop.\ VII.2.2]{RealLLC} and the variant of loc.\ cit.\ for $\A^{2, \an}/\G_m^{2, \an}$ and $\Z^2$-filtered vector spaces.
\end{proof}

\section{The Abel--Jacobi map}

\begin{defi}
    The \emph{Abel--Jacobi morphism} is defined as the map
    $$
    \AJ:\Div_E^1 \rightarrow \Pic_E^1
    $$
    sending $(\L,s)$ to $\L$, where we recall that $\Pic_E^1$ is the stack classifying degree $1$ line bundles on $X_E$.
\end{defi}

Recall that our goal is to prove the following result:

\begin{thm}
\label{thm:main}
The map $\AJ$ induces an equivalence
\begin{equation*}
    \Map(\Pic^1_E, */\G_m^\an)\xrightarrow{\cong} \Map(\Div^1_E, */\G_m^\an)
\end{equation*}
of mapping anima. In particular, any line bundle on $\Div^1_E$ is pulled back from $\Pic^1_E$.
\end{thm}

As mentioned earlier, our strategy to prove this result will be the following: First, one formally reduces to the case $E=\C$. For any line bundle $\L$ on $\Div^1$, we will then use the results of the previous section to construct a corresponding line bundle $\L^{(d)}$ on $\Div^d$ roughly given by $\L^{\boxtimes d}$. These line bundles will have the property that they are compatible under pullback along any map $\Div^d\rightarrow \Div^{d+1}$ given by multiplication with a chosen degree $1$ divisor. Thus, the $\L^{(d)}$ will assemble into a line bundle on some big colimit $\colim \Div^d$ and, choosing the degree $1$ divisors inducing the transition maps carefully, we will ensure that
\begin{equation*}
    \colim \Div^d\cong (\C^\infty\setminus \{0\})^\Betti\,/\,\C^{\times, \la}\;,
\end{equation*}
which implies the claim as $\C^\infty\setminus \{0\}$ is contractible.
\begin{lem}
\label{lem:symmetricpowers}
Let $E=\C$. For any $d\geq 1$, there is a canonical map
\begin{equation*}
    \Map(\Div^1, */\G_m^\an)\xrightarrow{\can} \Map(\Div^d, */\G_m^\an)\;.
\end{equation*}
Moreover, for any $\C$-point of $\Div^1$, there is an induced map $\Div^d\rightarrow\Div^{d+1}$ and the composition
\begin{equation*}
    \Map(\Div^1, */\G_m^\an)\xrightarrow{\can}\Map(\Div^{d+1}, */\G_m^\an)\rightarrow \Map(\Div^d, */\G_m^\an)
\end{equation*}
identifies with the canonical map from above.
\end{lem}
\begin{proof}
As $\Sym \Div^1\coloneqq \bigsqcup_{d\geq 0} \Sym^d \Div^1$ is the free commutative monoid stack on $\Div^1$, we have
\begin{equation*}
    \Map(\Div^1, */\G_m^\an)\cong \Map_{\mathrm{MonStk}}(\Sym \Div^1, */\G_m^\an)
\end{equation*}
by adjunction. In particular, precomposing with the inclusion 
\begin{equation*}
    \Div^d_{\HT\leq 1}\hookrightarrow \Sym^d \Div^1\hookrightarrow\Sym \Div^1
\end{equation*}
from \cref{prop:symdisohtleq1} yields a canonical map
\begin{equation*}
    \Map(\Div^1, */\G_m^\an)\rightarrow \Map(\Div^d_{\HT\leq 1}, */\G_m^\an)\cong \Map(\Div^d, */\G_m^\an)\;,
\end{equation*}
where the last isomorphism is due to \cref{prop:mappingspaces}. The claimed compatibility follows by construction as the maps $\Div^d\rightarrow */\G_m^\an$ we obtain arise from a map of commutative monoid stacks $\Sym \Div^1\rightarrow */\G_m^\an$.
\end{proof}

We stay in the case $E=\C$ and consider the two degree $1$ divisors $V(x)$ and $V(y)$ on $\P^1_\C$. These induce transition maps
\begin{equation*}
    x, y: \Div^d\rightarrow \Div^{d+1}
\end{equation*}
for any $d\geq 1$. 

\begin{lem}
\label{lem:colimcinfty}
Let $E=\C$. Then there is an isomorphism
\begin{equation*}
    \colim_{x, y} \Div^d\cong (\C^\infty\setminus \{0\})^\Betti\,/\,\C^{\times, \la}\;.
\end{equation*}
\end{lem}
\begin{proof}
We note that $x$ and $y$ even induce maps
\begin{equation*}
    x, y: \BC(\O(d))\rightarrow\BC(\O(d+1))
\end{equation*}
for $d\geq 1$ and hence \ref{prop:divd} reduces us to showing that
\begin{equation}
\label{eq:aj-colimcinfty}
    \colim_{x, y} \BC(\O(d))\cong \C^{\infty, \Betti}\;.
\end{equation}
By \cref{prop:bcspaces}, the $x$ maps identify with the maps
\begin{equation*}
    \A^{1, \an}\times \C^{d-1, \Betti}\times \A^{1, \an}\rightarrow \A^{1, \an}\times \C^{d, \Betti}\times \A^{1, \an}
\end{equation*}
given by the identity on the last factor and the right shift on the first $d$ factors; in particular, on the factor $\A^{1, \an}$, the map $x$ is given by the canonical projection $\A^{1, \an}\rightarrow \C^\Betti$. Thus, we conclude that
\begin{equation*}
    \colim_x \BC(\O(d))\cong \C^{\infty, \Betti}\times \A^{1, \an}
\end{equation*}
and it remains to take the colimit along the $y$-maps. However, the map
\begin{equation*}
    y: \C^{\infty, \Betti}\times \A^{1, \an}\rightarrow \C^{\infty, \Betti}\times \A^{1, \an}
\end{equation*}
is now given by the left shift $\C^{\infty, \Betti}\times\A^{1, \an}\rightarrow \C^{\infty, \Betti}$, from which we conclude the isomorphism (\ref{eq:aj-colimcinfty}), as desired.
\end{proof}

Finally, we can put all the ingredients together and prove \cref{thm:main}.

\begin{proof}[Proof of \cref{thm:main}]
We first note that we may reduce to the case $E=\C$. Indeed, this follows from the fact that $\Div^1_\R\cong \Div^1_\C/(\Z/2)$ and $\Pic^1_\R\cong \Pic^1_\C/(\Z/2)$, see \cite[Prop.\ VI.4.2]{RealLLC}.

In the case $E=\C$, we note that \cref{lem:symmetricpowers} and \cref{lem:colimcinfty} produce a map
\begin{equation*}
\begin{split}
    \Map(\Div^1, */\G_m^\an)\xrightarrow{\can} \lim_{x, y} \Map(\Div^d, */\G_m^\an)&\cong \Map((\C^\infty\setminus \{0\})^\Betti/\C^{\times, \la}, */\G_m^\an) \\
    &\cong \Map(*/\C^{\times, \la}, */\G_m^\an)\;,
\end{split}
\end{equation*}
where the last isomorphism is due to \cref{lem:picbetti}. However, note that this map has a retraction induced by
\begin{equation*}
    \lim_{x, y}\Map(\Div^d, */\G_m^\an)\rightarrow \Map(\Div^1, */\G_m^\an)
\end{equation*}
and, under the isomorphism $\Pic^1\cong */\C^{\times, \la}$ from \cref{prop:pic}, this retraction identifies with precomposition with the Abel--Jacobi map. In other words, we conclude that the map
\begin{equation}
\label{eq:aj-precompaj}
    \Map(\Pic^1, */\G_m^\an)\rightarrow \Map(\Div^1, */\G_m^\an)
\end{equation}
induced by the Abel--Jacobi map is the projection onto a direct summand. However, picking any $\C$-point of $\BC(\O(1))\setminus \{0\}$, e.g.\ the one given by the function $x$ on $\P^1_\C$, we obtain a map
\begin{equation*}
    \Pic^1\cong */\C^{\times, \la}\rightarrow \BC(\O(1))\setminus \{0\}\,/\,\C^{\times, \la}\cong \Div^1
\end{equation*}
precomposition with which yields a retraction of (\ref{eq:aj-precompaj}). Thus, we deduce that (\ref{eq:aj-precompaj}) must be an isomorphism, which is what we wanted to prove.
\end{proof}

Finally, let us explain how to deduce local class field theory for $E$, i.e. that $W_E^\ab\cong E^\times$.

\begin{proof}[Proof of \cref{cor:intro-cft}]
By Pontryagin duality, it suffices to show that any continuous complex character of $W_E$ factors through $E^\times$. To see this, note that the presentation
\begin{equation*}
    \Div^1_E\cong (\A^{2, \an}\setminus \{0\})\,/\,W_E^\la
\end{equation*}
from \cite[Cor.\ VI.4.5]{RealLLC} shows that any such character gives rise to a line bundle on $\Div^1_E$ using \cref{Findimrep}. However, by \cref{thm:main}, any such line bundle is pulled back from $\Pic^1_E\cong */E^{\times, \la}$ and line bundles on this stack are characters of $E^\times$ by another application of \cref{Findimrep}.
\end{proof}

\bibliographystyle{alpha}
\bibliography{References}
\end{document}